\documentclass[12pt,a4paper,reqno]{amsart}

\usepackage[margin=2.8cm,footskip=1cm]{geometry}
\usepackage[utf8]{inputenc}
\usepackage[T1]{fontenc}
\usepackage[english]{babel}
\usepackage{amsmath}
\usepackage{amsfonts}
\usepackage{amssymb}
\usepackage{amsthm}
\usepackage{ucs}
\usepackage{graphicx}
\usepackage{xcolor}
\usepackage{dsfont}
\usepackage{tikz}
\usepackage{wasysym}
\usepackage{physics}
\usepackage{mathtools}
\usepackage{xifthen}
\usepackage[textwidth=2.5cm,textsize=tiny]{todonotes}
\usepackage{enumitem}
\usepackage{stmaryrd}
\usepackage{constants}
\usepackage[ruled]{algorithm2e}
\usepackage{tikz,pgfplots}
\usepackage[hidelinks]{hyperref}
\usepackage[font={small}]{caption}

\makeatletter
\newcommand{\pushright}[1]{\ifmeasuring@#1\else\omit\hfill$\displaystyle#1$\fi\ignorespaces}
\newcommand{\pushleft}[1]{\ifmeasuring@#1\else\omit$\displaystyle#1$\hfill\fi\ignorespaces}
\makeatother

\newcommand{\normsup}[1]{\left\|#1\right\|_{\scriptscriptstyle\infty}}

\newcommand{\R}{\mathbb{R}}

\newcommand{\rmd}{\mathrm{d}}
\newcommand{\rmi}{\mathrm{i}}

\newcommand{\calN}{\mathcal{N}}
\newcommand{\calP}{\mathcal{P}}
\newcommand{\calQ}{\mathcal{Q}}
\newcommand{\calX}{\mathcal{X}}

\theoremstyle{plain}
\newtheorem{theorem}{Theorem}[section]
\newtheorem{lemma}[theorem]{Lemma}

\newtheorem{corollary}[theorem]{Corollary}

\newtheorem{remark}{Remark}[section]

\theoremstyle{definition}

\newtheorem{obs}{Observation}

\author{S\'{e}bastien Ott}
\address{D\'epartement de Math\'ematiques, Universit\'e de Fribourg,
	Chemin du Mus\'ee 23, 1700 Fribourg, Switzerland}
\email{ott.sebast@gmail.com}

\date{\today}

\title{A note on the renormalization group approach to the Central Limit Theorem}

\begin{document}

\begin{abstract}
	A proof of the Central Limit Theorem using a renormalization group approach is presented. The proof is conducted under a third moment assumption and shows that a suitable renormalization group map is a contraction over the space of probability measures with a third moment. This is by far not the most optimal proof of the CLT, and the main interest of the proof is its existence, the CLT being the simplest case in which a renormalization group argument should apply. None of the tools used in this note are new. Similar proofs are known amongst expert in limit theorems, but explicit references are not so easy to come by for non-experts in the field.
\end{abstract}

\maketitle

%%%%%%%%%%%%%%%%%%%%%%%%%%%%%%%%%%%%%%%%%%%%%%%%%%%%%%%%%%%%%%%%%%%%%%%%%%%%%%%%%%%%%%%%%%%%%%%%%%%

\section*{A word from the author}

I am neither a renormalization group expert, nor a limit theorem expert. So it is perfectly possible that the results described this note appeared somewhere else \footnote{Indeed, after the second version of this note appeared on the arXiv, I received a reference to a paper, ~\cite{Hamedani+Walter-1984}, containing the proof presented in section~\ref{sec:Banach_FP}. It is historically interesting to note that they introduce (and use) the Fourier based metrics: the paper is from 1984 and the Fourier based metrics where believed to have been introduced in~\cite{Gabetta+Toscani+Wennberg-1995}, more than 10 years later.}, or that I missed relevant references to the ``Renormalization group approach to CLT'' story. In both cases, I would be extremely grateful to receive pointers towards the relevant references. After the appearance of (the first version of) this note on the arXiv, I received several emails with references and/or comments. The updated version is motivated by the content of those mails, I hope that the enhanced bibliography (which is still far from being exhaustive) will help other people interested in these questions as much as it helped me. Many thanks to all the people who took the time to send me information!

\section*{Introduction/bibliographical review}

This note is about providing a simple renormalization group style proof of the classical Central Limit Theorem. The CLT is likely to be the most well known feat of probability theory, so I do not intend to say anything new about it. The purpose of this work is more to perform a ``sanity check'' for the renormalization group approach: if one can not get the method to work in the simplest instance it should apply to, one has little chances of success in much more involved situations. This note is not the first result about the renormalization road to the CLT, and one can distinguish two general approaches: a ``Lyapunov'' method, consisting of solving the fixed point equation and, more or less explicitly, finding a suitable Lyapunov function for the discrete dynamical system \(\mu_{n+1}= T\mu_n\) (where \(T\) is the renormalization map); and a ``Banach'' method, consisting in finding a subset of the set of probability measures \(\Omega\), and a distance \(d\) on \(\Omega\) such that 1) \((\Omega,d)\) is complete 2) the renormalization operation is a contraction on \((\Omega,d)\). On the one side, the ``Banach'' method yields more quantitative results and \emph{one does not need to solve the fixed point equation} to prove convergence towards something. On the other side, the ``Lyapunov'' method is more flexible and allows for weaker conditions. Section~\ref{sec:Banach_FP} contains an example of what I called a ``Banach'' approach, while Section~\ref{sec:Lypunov_proof} presents a ``Lyapunov'' approach.

\subsection*{``Lyapunov'' method}
On the ``perturbative side'' (perturbation around the fixed point), one can find the articles~\cite{Jona-Lasinio-1975, Calvo+Cuchi+Esteve+Falceto-2010, Li+Sinai-2014} which deal with the transform of equation~\eqref{eq:renorm_transform} (and its natural generalizations) that have as fixed points stable laws. Suitable neighbourhood of the fixed points are then shown to contract to the fixed point. Reviews/introductory texts (mostly focused on the Gaussian fixed point) can be found in~\cite[pages 131-132]{Sinai-1992},~\cite{Jona-Lasinio-2001},~\cite[Section 10.3]{Koralov+Sinai-2007}. A non-rigorous discussion on that problem can be found in~\cite{Amir-2020}. The same set of ideas has been used to prove CLT for dependent fields, as this is not the main topic of this note, I will only mention a few early references:~\cite{Bleher+Sinai-1973} deal with hierarchical spin models, their method influenced several of the papers mentioned previously;~\cite{Gallavotti+Jona-Lasinio-1975, Cassandro+Jona-Lasinio-1978} contain reviews of CLT for the ``magnetization'' of (weakly) dependent lattice random fields (the second also provides a general discussion about renormalization and CLT).

Another route has also been studied: the entropy can be used as a Lyapunov function for the dynamical system underlying the renormalization process, see~\cite{Linnik-1959} for the CLT using entropy,~\cite{Carlen+Soffer-1991} for an ``Lyapunov function'' proof, and~\cite{Johnson-2004} for a book on the topic.

\subsection*{``Banach'' method}
A Banach fixed point argument can be found in~\cite{Neininger+Ruschendorf-2004} (see the slides~\cite{Neininger-2007} for a discussion of the application to the CLT). The proof there is very similar to the one of Section~\ref{sec:Banach_FP}: use of ideal metric (Zolotarev metric -introduced in~\cite{Zolotarev-1976,Zolotarev-1978}-, in~\cite{Neininger-2007} versus Fourier based metric here) to prove contraction properties of the renormalization map.

\subsection*{Present note}
This note contain an example of each methods. As said above, the approaches taken here are not new: the renormalization group map is the one of \cite{Jona-Lasinio-1975} and is probably older, the contraction principle of Section~\ref{sec:Banach_FP} is included in the paper~\cite{Goudon+Junca+Toscani-2002}: equations (6), (7) are morally the contraction used here (which is a property of certain ideal metrics, see Section~\ref{sec:conclusion}). The argument presented in~\cite{Neininger+Ruschendorf-2004,Neininger-2007} is the same as the one of Section~\ref{sec:Banach_FP} but with a different metric (see Section~\ref{sec:conclusion}).

The argument of Section~\ref{sec:Lypunov_proof} is a simplified version of an argument that has been communicated to me by Jiwoon Park. It provides an example of the larger flexibility of the ``Lyapunov'' method, and applies with only a second moment condition.

\subsection*{General comment}
It is worth noting that a contraction principle is underlying most approaches to the CLT, but the latter is usually formulated as follows: one has a sequence of operators, \((T_n)_{n\geq 1}\) (which acts by \(n\)-fold convolution and rescaling by \(\sqrt{n}\)), having the normal distribution as fixed point, and having ``law-dependent contraction constant'' going to \(0\) with \(n\), rather than iterating a fixed transform with uniform (over probability distribution) contraction constant.

Moreover, the proof of Section~\ref{sec:Lypunov_proof} can be seen as a slightly convoluted re-writing of the usual proof of the CLT through pointwise convergence of characteristic function, I nevertheless find it a nice example of global asymptotic stability of the fixed point of the renormalization process. The argument of Section~\ref{sec:Banach_FP} is more quantitative but more restrictive.

\section{Renormalization road to CLT: a ``Banach'' result}
\label{sec:Banach_FP}

As nothing is really new in the content of this text, I will work under a third moment assumption. The latter can be relaxed to a \(2+\epsilon\) moment assumption: the distance used (denoted \(\rmd_3\) in the text) is a particular case of the Fourier-based distances introduced in~\cite{Gabetta+Toscani+Wennberg-1995}. Replacing this distance by its \(\rmd_{2+\epsilon}\) version handles the extension (\(3\) is chosen for aesthetic reasons and to slightly simplify appendix~\ref{app:Metric_structure}). A review on these metrics and their applications can be found in~\cite{Carrillo+Toscani-2007}.

\subsection{Framework}

Probability measures on \(\R\) will be denoted \(\nu,\mu\), the expectation under \(\nu\) will be denoted \(E_{\nu}\). Inside expected value, \(X\) will denote a random variable of the relevant law, and \((X,Y)\) a random vector of law \(\nu\otimes \mu\) when this case is considered. Let \(\calP_r=\calP_r(\R)\) be the set of probability measures on \(\R\) with finite \(r\)th absolute moment. Denote
\begin{equation*}
	\calQ_3 := \{\nu\in \calP_3:\ E_{\nu}(X) = 0,\ E_{\nu}(X^2) =1 \},
\end{equation*}the set of centred, reduced probability measures with a third absolute moment. Equip \(\calQ_{3}\) with the Fourier-based distance
\begin{equation}
	\rmd_3(\nu,\mu) = \sup_{\xi\in \R^*} \frac{\big| \varphi_{\nu}(\xi) - \varphi_{\mu}(\xi) \big|}{|\xi|^3}
\end{equation}where \(\R^* = \R\setminus \{0\}\), and
\begin{equation}
	\varphi_{\nu}(\xi) = E_{\nu}\big(e^{\rmi X\xi}\big),
\end{equation}is the characteristic function of \(\nu\) (\cite{Carrillo+Toscani-2007} defines it with a \(-\) sign in the exponential).

\begin{lemma}
	\label{lem:Finite_metric_space_struct}
	\(\rmd_3\) is a finite distance on \(\calQ_3\), and convergence in \(\rmd_3\) implies weak convergence.
\end{lemma}
This result can be imported form~\cite{Gabetta+Toscani+Wennberg-1995, Carrillo+Toscani-2007}, but a proof is included in Appendix~\ref{app:Metric_structure}. The goal will be to study convergence towards a normal distribution \(\calN(0,1)\). Denote \(\gamma\) the normal law:
\begin{equation}
	d\gamma(x) = \frac{1}{\sqrt{2\pi}} e^{-x^2/2} dx.
\end{equation}One obviously has \(\gamma \in \calQ_3\).

\begin{remark}
	It is worth mentioning that \((\calQ_3,\rmd_3)\) is not expected to be a complete metric space, but asking for \(3+\epsilon\) moments rather than \(3\) guarantees completeness (see~\cite[Proposition 2.7]{Carrillo+Toscani-2007} and the follow-up discussion).
\end{remark}

\subsection{Renormalization transform}

Denote \(\nu*\mu\) the convolution of \(\nu\) and \(\mu\) (the law of \(X+Y\) where \(X\sim \nu, Y\sim \mu\) are independent). Also denote \([\nu]_{\lambda}\) the law of \(\lambda X\), where \(X\sim \nu\). One then consider the following renormalization transformation on probability measures: \(T\nu\) is the law of \(2^{-1/2}(X+Y)\) where \(X,Y\) are independent random variables of law \(\nu\):
\begin{equation}
\label{eq:renorm_transform}
	E_{T\nu}(f) = E_{\nu\otimes \nu}\Big(f\big((X+Y)/\sqrt{2}\big)\Big).
\end{equation}In words: \(T\) maps \(\nu\) to the renormalized convolution of \(\nu\) with itself (\(T\nu = [\nu^{*2}]_{2^{-1/2}}\)). Taking the sum is the ``coarse graining'' part of a renormalization step and dividing by \(\sqrt{2}\) is the ``rescaling'' part.

\begin{lemma}
	\label{lem:stability}
	If \(\nu\in \calQ_3\), then \(T\nu\in \calQ_3\).
\end{lemma}
\begin{proof}
	From the definition, one has that if \(\nu\) has a second moment, and \(E_{\nu}(X) = 0\), then \(E_{T\nu}(X) =0\), and \(E_{T\nu}(X^2) = E_{\nu}(X^2)\). Then,
	\begin{equation*}
		E_{T\nu}(|X|^3) = \frac{1}{2^{3/2}}E_{\nu\otimes \nu}(|X+Y|^3)\leq \frac{16}{2^{3/2}} E_{\nu}(|X|^3)<\infty.
	\end{equation*}
\end{proof}

The goal of this note is to study the CLT, so the main interest of this transformation is
\begin{equation}
	\label{eq:fixed_pt}
	T\gamma = \gamma,
\end{equation}which is a standard consequence of the stability of the Gaussian distribution.

The link with the CLT is as follows: if \(X_1,X_2,\dots\) form an i.i.d. sequence of law \(\nu\in \calQ_3\), then \(T^n \nu \) is the law of
\begin{equation*}
	\frac{1}{2^{n/2}}\sum_{k=1}^{2^n} X_k \equiv \frac{1}{\sqrt{N}}\sum_{k=1}^{N} X_k,
\end{equation*}where \(N= 2^n\).

\subsection{Contraction and CLT}

\begin{theorem}
	\label{thm:contraction_principle}
	The application \(T\) is a contraction on \((\calQ_3,\rmd_3)\) with contraction constant \(\leq 2^{-1/2}\). In particular (by~\eqref{eq:fixed_pt}),
	\begin{equation}
		\rmd_3(T^n \nu, \gamma) \leq 2^{-n/2} \rmd_3(\nu, \gamma).
	\end{equation}
\end{theorem}
\begin{proof}
	The proof is almost trivial. Let \(\mu,\nu\in \calQ_3\). First, notice that
	\begin{equation}
		\varphi_{T\nu}(\xi) = E_{\nu\otimes \nu}(e^{\rmi \xi (X+Y)/\sqrt{2}}) = \varphi_{\nu}(\xi/\sqrt{2})^2.
	\end{equation}
	 Then, writing the definition, one has
	\begin{align*}
		\rmd_3(T\nu, T\mu) &= \sup_{\xi\in \R^*} \frac{\big|\varphi_{\nu}(\xi/\sqrt{2})^2 - \varphi_{\mu}(\xi/\sqrt{2})^2 \big|}{2^{3/2}|\xi/\sqrt{2}|^3}\\
		&= 2^{-3/2} \sup_{\xi\in \R^*} \frac{\big|\varphi_{\nu}(\xi)^2 - \varphi_{\mu}(\xi)^2 \big|}{|\xi|^3}\\
		&= 2^{-3/2} \sup_{\xi\in \R^*} \frac{\big|\varphi_{\nu}(\xi) - \varphi_{\mu}(\xi) \big|}{|\xi|^3}\big|\varphi_{\nu}(\xi) + \varphi_{\mu}(\xi) \big|\\
		&\leq 2^{-1/2} \sup_{\xi\in \R^*} \frac{\big|\varphi_{\nu}(\xi) - \varphi_{\mu}(\xi) \big|}{|\xi|^3} = 2^{-1/2} \rmd_3(\nu,\mu),
	\end{align*}as \(\normsup{\varphi_{a}} = 1\) for any probability measure \(a\).
\end{proof}

\begin{corollary}[Central Limit Theorem]
	\label{cor:CLT}
	Let \(\nu\in \calQ_3\). Let \(X_1,X_2,\dots\) be an i.i.d. sequence of law \(\nu\). Then,
	\begin{equation}
		\frac{1}{\sqrt{N}}\sum_{k=1}^N X_k \xrightarrow{N\to\infty} \calN(0,1),
	\end{equation}where the convergence is \(\rmd_3\) distance (and therefore also in law).
\end{corollary}
The claim along the sequence \(N=1,2,4,8,\dots\) (or any geometric sequence) follows directly from Theorem~\ref{thm:contraction_principle}, Lemma~\ref{lem:stability}, and the fact that \(\rmd_3\) is a \emph{finite} distance on \(\calQ_3\) (by Lemma~\ref{lem:Finite_metric_space_struct}). Extending this to arbitrary sequences can be done in several ways, which shall not be exposed here. It is worth stressing out that extending the result to arbitrary sequences is never simpler than using stability of Gaussian and scaling+convolution properties of \(\rmd_3\) to obtain the CLT directly. Indeed,
\begin{equation}
\label{eq:ideal metric}
\begin{gathered}
	\rmd_3(\nu_1*\nu_2, \mu_1*\mu_2) \leq \rmd_3(\nu_1, \mu_1)+\rmd_3(\nu_2, \mu_2),\\
	\rmd_3([\nu]_{\lambda},[\nu]_{\lambda}) \leq \lambda^3 \rmd_3(\nu,\mu),
\end{gathered}
\end{equation}where \([\nu]_{\lambda}\) is the law of \(\lambda X, X\sim \nu\). The first point follows from \(|ab-cd|\leq |a-c||b| + |b-d||c|\), and the definition. Therefore, by the properties of the Gaussian,
\begin{equation*}
	\rmd_3([\nu^{*n}]_{n^{-1/2}}, \gamma) = \rmd_3([\nu^{*n}]_{n^{-1/2}}, [\gamma^{*n}]_{n^{-1/2}}) \leq \frac{1}{n^{3/2}} \rmd_3(\nu^{*n}, \gamma^{*n})\leq \frac{\rmd_3(\nu, \gamma)}{n^{1/2}}.
\end{equation*}

\section{A ``Lyapunov proof'' under a second moment condition}
\label{sec:Lypunov_proof}

As said in the introduction, the following proof is a simplified version of an argument that has been sent to me by Jiwoon Park. It is presented here with his kind permission.

Let
\begin{gather*}
	\calQ_2 = \{\nu\in \calP_2:\ E_{\nu}(X) =0,\ E_{\nu}(X^2) = 1 \},\\
	\rmd_2(\nu,\mu) = \sup_{\xi\in \R^*} \frac{|\varphi_{\nu}(\xi)- \varphi_{\mu}(\xi)|}{\xi^2}.
\end{gather*}Then, \(\rmd_2\) is a (finite) distance on \(\calQ_2\) (which metrizes weak convergence). The problem with \(\rmd_2\) is that it is only \(2\)-ideal (see Section~\ref{sec:conclusion}), so the argument of the previous Section can not be used with this metric (but still gives continuity of \(T\)). Yet one can still use it as a Lyapunov function.

\begin{theorem}
	\label{thm:Lyaponuv_contraction}
	Let \(V:\calQ_2\to \R_+\) be defined by \(V(\nu) = \rmd_2(\nu,\gamma)\). Then, \(V\) satisfies
	\begin{enumerate}
		\item \label{thm:Lyap:item:positivity} \(V(\gamma) =0\), and \(V(\nu)>0\) for any \(\nu\in \calQ_2\setminus \{\gamma\}\).
		\item \label{thm:Lyap:item:continuity} \(V\) is continuous on \((\calQ_2,\rmd_2)\).
		\item \label{thm:Lyap:item:bnded} \(V\) has bounded level sets.
		\item \label{thm:Lyap:item:monotonicity} \(V(T\nu)< V(\nu)\) for every \(\nu\in \calQ_2\setminus \{\gamma\}\).
	\end{enumerate}
\end{theorem}
\begin{proof}
	Items~\ref{thm:Lyap:item:positivity},~\ref{thm:Lyap:item:continuity}, and~\ref{thm:Lyap:item:bnded} are obvious as \(\rmd_2\) is a distance that metrizes weak convergence. Only item~\ref{thm:Lyap:item:monotonicity} requires some care. Let \(\nu\neq \gamma\). Then,
	\begin{equation*}
		V(T\nu) = \sup_{\xi\in \R^*} \frac{|\varphi_{\nu}(\xi)- \varphi_{\gamma}(\xi)|}{\xi^2}\frac{|\varphi_{\nu}(\xi) + \varphi_{\gamma}(\xi)|}{2}\leq \sup_{\xi\in \R^*} \frac{|\varphi_{\nu}(\xi)- \varphi_{\gamma}(\xi)|}{\xi^2}\frac{1+e^{-\xi^2/2}}{2}.
	\end{equation*}Now, by second order Taylor expansion,
	\begin{equation*}
		\varphi_{*}(\xi) = 1-\frac{\xi^2}{2} + \xi^2h_*(\xi)
	\end{equation*}with \(h_*(\xi) \xrightarrow{\xi\to 0} 0\). Let \(a = V(\nu) = \rmd_2(\nu,\gamma)>0\), and \(b>0\) be such that \linebreak \(\sup_{|\xi|<b} \max(|h_{\nu}(\xi)|, |h_{\gamma}(\xi)|)< \frac{a}{4}\). Then, the last term of the last display is less than or equal to
	\begin{multline*}
		\max\Big(\sup_{|\xi|\geq b} \frac{|\varphi_{\nu}(\xi)- \varphi_{\gamma}(\xi)|}{\xi^2}\frac{1+e^{-b^2/2}}{2} , \sup_{\xi\in (-b,b)\setminus \{0\}} |h_{\nu}(\xi)- h_{\gamma}(\xi)|\Big)\leq \\
		\leq \max\Big(\frac{1+e^{-b^2/2}}{2} a , \frac{a}{2} \Big) < a,
	\end{multline*}proving the last item.
\end{proof}

The immediate consequence is
\begin{corollary}
	Let \(\nu\in \calQ_2\). Let \(X_1,X_2,\dots\) be an i.i.d. sequence of law \(\nu\). Then,
	\begin{equation}
		\frac{1}{2^{n/2}}\sum_{k=0}^{2^n} X_i \xrightarrow{n\to\infty} \calN(0,1).
	\end{equation}
\end{corollary}
\begin{proof}
	The proof is the standard proof that global asymptotic stability is implied by the existence of a suitable Lyapunov function. It is included for the reader convenience. Let \(V(\nu) = \rmd_2(\nu,\gamma)\). The claim is equivalent to \(V(T^n \nu) \to 0\) as \(n\to \infty\) for any \(\nu\in \calQ_2\). Fix \(\nu\in\calQ_2\). Define
	\begin{equation*}
		D = \{\mu\in \calQ_2:\ 0\leq V(\mu)\leq V(\nu) \}.
	\end{equation*}By items~\ref{thm:Lyap:item:bnded},\ref{thm:Lyap:item:continuity} of Theorem~\ref{thm:Lyaponuv_contraction}, \(D\) is compact. Moreover, as \(V(T\mu)\leq V(\mu)\), \(\nu_n \equiv T^n \nu\in D\) for all \(n\geq 0\). The sequence \(V(\nu_n)\) is then a bounded decreasing sequence, denote \(a\geq 0\) its limit. By continuity of \(V\), every cluster point, \(\nu_*\), of \((\nu_n)_n\) has \(V(\nu_*) = a\). Now, if \((\nu_{n_k})_k\) is a converging subsequence of \((\nu_{n})_n\), so is \((\nu_{n_k+1})_k\) (by continuity of \(T\)). Moreover, the limit of \((\nu_{n_k+1})_k\) is \(T\nu_*\) where \(\nu_*\) is the limit of \((\nu_{n_k})_k\). So, \(a = V(T\nu_*) \leq V(\nu_*) = a\) with equality only if \(\nu_*=\gamma\) (by Theorem~\ref{thm:Lyaponuv_contraction}, item~\ref{thm:Lyap:item:monotonicity}). So either \(\nu_* = \gamma\) or \(a=0\) which imply each other.
\end{proof}

\section{Additional remarks}
\label{sec:conclusion}

\subsection*{About the metric}
The metrics \(\rmd_3,\rmd_2\) (and their generalizations \(\rmd_s\), see~\cite{Goudon+Junca+Toscani-2002}) as well as the Zolotarev metric mentioned in the introduction and used in~\cite{Neininger+Ruschendorf-2004, Neininger-2007}, belong to a class of metric on probability measures which all allow the same type of argument. These metric satisfy
\begin{gather}
\label{eq:ideal_convolution}
	d(\nu*\eta,\mu*\eta)\leq d(\nu,\mu),\\
\label{eq:ideal_scaling}
	d([\nu]_{\lambda},[\mu]_{\lambda}) \leq \lambda^{s} d(\nu,\mu),
\end{gather}for some \(s>0\). Such a metric with equality in~\eqref{eq:ideal_scaling} is call \(s\)\emph{-ideal}. Suppose that the metric is defined as a supremum over a suitable class of test function of weighted expectation difference:
\begin{equation}
\label{eq:zeta_struct}
	d(\nu,\mu) = \sup_{f\in F} w(f) \big| E_{\nu}(f) - E_{\mu}(f) \big|,
\end{equation}such metric is apparently said to have a \emph{\(\zeta\)-structure}. Suppose that one can construct a metric \(d\) satisfying~\eqref{eq:ideal_convolution},~\eqref{eq:ideal_scaling} with \(s>2\), and~\eqref{eq:zeta_struct}. Further suppose that convergence in the topology induced by \(d\) implies weak convergence. Then one can perform the same argument, the key being
\begin{align*}
	d(\nu*\nu,\mu*\mu) &= \sup_{f\in F} w(f) \big| E_{\nu\otimes\nu}(f(X+Y)) - E_{\mu\otimes\mu}(f(X+Y)) \big|\\
	&\leq \sup_{f\in F} w(f) \big| E_{\nu\otimes\nu}(f(X+Y)) - E_{\nu\otimes\mu}(f(X+Y))\big| \\
	&\quad + \big|E_{\nu\otimes\mu}(f(X+Y)) - E_{\mu\otimes\mu}(f(X+Y)) \big|\\
	&\leq 2 d(\nu, \mu).
\end{align*}
It is also clear that the same type of argument as in Section~\ref{sec:Banach_FP} with \(s=2\) will not work so easily, so a different idea is required. Limitations of the method are discussed in~\cite{Neininger+Ruschendorf-2004}. One can also wonder if a ``Banach'' proof can even exist. An argument going in this direction is~\cite{Bessaga-1959}: if \(T:\calX\to\calX\) is such that \(T,T^2,T^3,\dots\) have a unique fixed point, one can construct a metric on \(\calX\) such that \(T\) is a contraction (with arbitrary contraction constant). But the argument is non-constructive, uses the axiom of choice, and the metric has no chance of metrizing weak topology. I would tend to think that the general claim is wrong but that the following weaker claim could true: one can find a sequence \(\calQ_2^{(1)}\subset \calQ_2^{(2)}\subset \dots\) converging to \(\calQ_2\) (i.e.: \(\nu\in \calQ_2\) implies there is \(n\geq 1\) such that \(\nu\in \calQ_2^{(n)}\)) and a sequence of metrics \(d_1,d_2,\dots\) on those spaces such that \(T\calQ_2^{(n)}\subset \calQ_2^{(n)}\) and \(T\) is a contraction on \((\calQ_2^{(n)}, d_n)\) (and \((\calQ_2^{(n)}, d_n)\) is complete, and \(d_n\) metrizes weak convergence on \(\calQ_2^{(n)}\)...). Any comment on this question would be welcome!

\section*{Acknowledgements}

Thanks to Nicolas Curien for suggesting me to contact Ralph Neininger, and to Ralph Neininger for pointers to~\cite{Neininger+Ruschendorf-2004,Neininger-2007}, and comments on the general picture. Also thanks to Jiwoon Park for sending me the argument which led to Section~\ref{sec:Lypunov_proof}. Finally, thanks to all the people who send me information and comments after the first version went online.

The author is supported by the Swiss NSF grant 200021\_182237 and is a member of the NCCR SwissMAP.

\appendix

\section{Fourier based distance}
\label{app:Metric_structure}

This Section contains a proof of Lemma~\ref{lem:Finite_metric_space_struct}. The proof is by no mean new, it is included for the reader convenience.

First, note that \(\rmd_{3}\) is a distance: symmetry is obvious, and separation follows from the fact that two probability measures are the same if and only if they have the same characteristic function. Triangular inequality follows from triangular inequality for the absolute value, and from \(\sup f+g \leq \sup f + \sup g\).

Then, show that \(\rmd_{3}\) is finite on \(\calQ_3\times \calQ_3\). Take \(\nu,\mu\in \calQ_3\). As \(\nu,\mu\) have a third moment, \(\varphi_{\nu},\varphi_{\mu}\) are three times continuously differentiable and they admit a Taylor expansion at \(0\):
\begin{equation}
	\varphi_*(\xi) = 1 - \frac{\xi^2}{2} - \rmi\frac{E_{*}(X^3) \xi^3}{6} + h_*(\xi) \xi^3,
\end{equation}with \(h_*(\xi) \xrightarrow{\xi\to 0} 0\) bounded uniformly over \([-1,1]\) (as it is continuous over \(\R\)), \(*\in \{\nu,\mu\}\). So,
\begin{equation*}
	\rmd_3(\nu,\mu) \leq 6^{-1}\sup_{0<|\xi|<1} \big|- \rmi E_{\nu}(X^3) + 6h_{\nu}(\xi)  + \rmi E_{\mu}(X^3) - 6h_{\mu}(\xi) \big| + 2<\infty.
\end{equation*}

Finally, convergence in \(\rmd_3\) implies pointwise convergence of characteristic functions, as well as continuity of the limit at \(0\), which is equivalent to weak convergence by Levy's continuity theorem.

\bibliographystyle{plain}
\bibliography{BIGbib}

\end{document}